\documentclass[12pt,a4paper]{article}
\usepackage[T1]{fontenc}
\usepackage[centertags]{amsmath}
\usepackage{amsfonts}
\usepackage{amssymb}
\usepackage{amsthm}
\usepackage{newlfont}
\usepackage{indentfirst}
\usepackage{float}
\usepackage[inline]{enumitem}
\usepackage{url}

\DeclareMathOperator{\tr}{tr}
\DeclareMathOperator{\ch}{char}
\DeclareMathOperator{\M}{M}
\DeclareMathOperator{\GL}{GL}
\DeclareMathOperator{\SL}{SL}
\DeclareMathOperator{\GA}{GA}
\DeclareMathOperator{\SA}{SA}
\DeclareMathOperator{\IL}{IL}
\DeclareMathOperator{\IA}{IA}

\DeclareMathOperator{\Fix}{Fix}
\DeclareMathOperator{\Lin}{Lin}
\DeclareMathOperator*{\Id}{Id}

\newtheorem{tw}{{\sf Theorem}}[section]
\newtheorem{lem}[tw]{{\sf Lemma}}
\newtheorem{wn}[tw]{{\sf Corollary}}
\newtheorem{prop}[tw]{{\sf Proposition}}

\theoremstyle{definition}
\newtheorem{deff}[tw]{{\sf Definition}}

\newtheorem{uw}[tw]{{\sf Remark}}

\numberwithin{equation}{section}

\newcommand{\blank}{{\mspace{1mu}\cdot\mspace{1mu}}}

\title{Paradoxical decompositions of finite-dimensional non-Archimedean normed spaces}

\author{Kamil Orzechowski\\
University of Rzesz\'ow\\ 
Rejtana 16c\\
35-959 Rzesz\'ow, Poland\\
\texttt{kamilo@dokt.ur.edu.pl}}

\date{March 12, 2025}

\begin{document}

\maketitle

\renewcommand{\thefootnote}{}



\renewcommand{\thefootnote}{\arabic{footnote}}
\setcounter{footnote}{0}

\begin{abstract}
   We show that any normed space $(K^n,\lVert\blank\rVert)$, $n\ge 2$, over a field $K$ equipped with a nontrivial non-Archimedean valuation admits a paradoxical decomposition using four pieces with respect to the group of its affine isometries, provided that the norm $\lVert\blank\rVert$ is equivalent to the maximum norm.

   It follows that any finite-dimensional normed space $(X,\lVert\blank\rVert)$ with $\dim{X}\ge 2$ over a complete non-Archimedean nontrivially valued field $(K,\lvert\blank\rvert)$ is paradoxical using four pieces with respect to the group of its affine isometries.
\end{abstract}

\section{Introduction}

The study of paradoxical decompositions began with the works of Hausdorff \cite{Haus} and Banach, Tarski \cite{BT}. The latter two mathematicians proved that any ball in $\mathbb{R}^3$ can be partitioned into finite number of pieces in such a way that the pieces can be transformed using isometries of $\mathbb{R}^3$ to obtain two balls identical with the initial one \cite[Lemma 21]{BT}, i.e., any ball in $\mathbb{R}^3$ is {\em $G$-paradoxical}, where $G$ is the group of all isometries of $\mathbb{R}^3$. This famous theorem is often called the Banach--Tarski (or Hausdorff--Banach--Tarski) paradox (sometimes the name refers to the related ``strong'' result about $G$-equidecomposability of any two bounded subsets of $\mathbb{R}^3$ with nonempty interiors \cite[Theorem 24]{BT}).

There have been many refinements, generalizations and inspired results since the classical work of Banach and Tarski was published. The monograph \cite{W} is an excellent source covering most of these topics. We especially point out that the whole Euclidean space $\mathbb{R}^n$ is paradoxical with respect to the group of all isometries for any $n\ge 3$. Moreover, only four pieces are needed for such a paradoxical decomposition of $\mathbb{R}^n$ \cite[Corollary 6.8]{W}. In general, the existence of a paradoxical decomposition of a $G$-set $X$ using four pieces is equivalent to the existence of a free non-Abelian subgroup of $G$ that acts {\em locally commutatively} (i.e., with Abelian stabilizers) on $X$ \cite[Theorems 5.5 and 5.8]{W}.

\bigskip

In our recent paper \cite{my_preprint}, there were established several results concerning paradoxical decompositions of some subsets of the normed spaces $(K^n,\Vert\blank\rVert_{\infty})$ over a non-Archimedean valued field $(K,\lvert\blank\rvert)$, where $\lVert\blank\rVert_{\infty}$ denotes the maximum norm. In particular, if the valuation $\lvert \blank\rvert$ is nontrivial and $n\ge 2$, the space $(K^n,\Vert\blank\rVert_{\infty})$ is  paradoxical with respect to a certain group of affine isometries using five pieces \cite[Theorem 5.7]{my_preprint}. All balls, and all spheres that contain the origin are paradoxical with respect to certain groups of affine isometries using four pieces \cite[Theorems 5.1 and 5.4]{my_preprint}. As for the spheres not containing $\mathbf{0}$, the number of pieces needed is $5-(-1)^n$, depending on the dimension $n\ge 2$ \cite[Theorem 5.2]{my_preprint}.

\bigskip

Our main goal is to extend the above result to the case of $K^n$ equipped with any non-Archimedean norm equivalent to $\lVert\blank\rVert_{\infty}$. This is a meaningful improvement, since it contrasts with the Archimedean case \cite{Jarosz}, where any (real or complex) Banach space $(X,\lVert \blank\rVert)$ has an equivalent norm $\lVert\blank\rVert^{'}$ such that any surjective linear isometry of $(X,\lVert\blank\rVert^{'})$ is of the form $\lambda \Id_{X}$ for some scalar $\lambda$ with $\lvert \lambda\rvert=1$; so the isometry group of $(X,\lVert\blank\rVert^{'})$ is amenable, which excludes a paradoxical decomposition of $(X,\lVert\blank\rVert^{'})$.

The second improvement is to bring down the number of pieces required for some paradoxical decompositions in \cite{my_preprint} to four, which is optimal since four is the least number of pieces in any paradoxical decomposition.

\subsubsection*{Main results}

\begin{enumerate}
    \item  {\it Let $(K,\lvert\blank\rvert)$ be a non-Archimedean nontrivially valued field and $n\ge 2$. If $\Vert \blank \rVert$ is a non-Archimedean norm on $K^n$ equivalent to the norm $\lVert \blank\rVert_{\infty}$, then $K^n$, all balls and all nonempty spheres in $K^n$ are paradoxical with respect to the group of affine isometries of $(K^n,\lVert\blank\rVert)$ using four pieces (Theorem \ref{th:main}).}
    \item {\it Let $(K,\lvert\blank\rvert)$ be a trivially valued field that is not locally finite. Let $n\ge 2$ and $\lVert\blank\rVert$ be a non-Archimedean norm on $K^n$. If the norm $\lVert \blank \rVert$ takes less than $n$ nonzero values, then $(K^n,\lVert\blank\rVert)$ is paradoxical with respect to the group of its affine isometries using four pieces; otherwise $K^n$ is not paradoxical with respect to this group (Theorem \ref{trivial_val}).}
    \item {\it  Let $K$ be a field and $n\in \mathbb{N}$.
    If $K$ is locally finite or $n=1$, then $K^n$ is not paradoxical with respect to the group of all affine bijections of $K^n$ (Theorem \ref{locally_fin}).}
\end{enumerate}

\section{Preliminaries}

\subsection{Non-Archimedean valued fields and normed spaces}

A {\em valuation} on a field $K$ is any function $\lvert\blank\rvert\colon K\to[0,\infty)$ satisfying for all $x,y\in K$ the following conditions:

\begin{enumerate*}[label=(\arabic*)]
    \item $\lvert x \rvert=0 \Leftrightarrow x=0$,
    \item $\lvert xy \rvert=\lvert x \rvert\lvert y \rvert$,
    \item $\lvert x+y \rvert\leq \lvert x \rvert+\lvert y \rvert$.
\end{enumerate*}\\
The pair $(K,\lvert\blank\rvert)$ is then called a {\em valued field}.
If we replace the condition (3) with a stronger one (3')
$\lvert x+y \rvert\leq \max\{\lvert x \rvert,\lvert y \rvert\},$ the valuation $\lvert\blank\rvert$ will be called {\em non-Archimedean}.

If $(X,d)$ is a metric space, $x_0\in X$ and $r>0$, we put\\
$B[x_0,r]:=\{x\in X: d(x,x_0)\leq r\}$, $B(x_0,r):=\{x\in X: d(x,x_0)<r\}$ and $S[x_0,r]:=\{x\in X: d(x,x_0)= r\}.$

A metric $d$ on a set $X$ is called an {\em ultrametric} if it satisfies the {\em strong triangle inequality}:
$d(x,z) \leq \max\{d(x,y), d(y,z)\}$ for all $x,y,z \in X$. The pair $(X,d)$ is then called an {\em ultrametric space}. Any ultrametric space $(X,d)$ has the {\em isosceles property}: If $x,y,z \in X$ and $d(x,y)\neq d(y,z)$, then
$d(x,z) = \max\{d(x,y), d(y,z)\}$. It is noteworthy that, in an ultrametric space, any point belonging to a ball (open or closed) of a given radius may play the role of its center.

Each valuation $\lvert\blank\rvert$ on a field $K$ induces a metric $d\colon K\times K \to [0, \infty)$, $d(x,y):=\lvert x-y \rvert$. If $\lvert\blank\rvert$ is non-Archimedean, $d$ is an ultrametric. We say that $(K,\lvert\blank\rvert)$ is {\em complete} if $(K,d)$ is a complete metric space. If $d_1$, $d_2$ are metrics on $K$ induced by non-Archimedean valuations $\lvert\blank\rvert_1$, $\lvert\blank\rvert_2$ respectively, they induce the same topology on $K$ if and only if $\lvert\blank\rvert_1$ and $\lvert\blank\rvert_2$ are {\em equivalent}, i.e., $\lvert\blank\rvert_2=(\lvert\blank\rvert_1)^c$ for some $c>0$ \cite[Theorem 1.2.2]{PG}.

Let $(K,\lvert\blank\rvert)$ be a fixed non-Archimedean valued field. We introduce the following notation:
$D_K= B[0,1]$, $\mathfrak{m}_K= B(0,1)$, and $D_{K}^{*} = S[0,1].$
The subset $D_K$ is a subring of $K$ called the {\em ring of integers} of $K$, $\mathfrak{m}_{K}$ is its unique maximal ideal, and $D_{K}^*$ is the group of units (i.e., invertible elements) of $D_K$ (cf. \cite[Lemma 1.2]{Schneider}). The quotient ring $k:=D_K/\mathfrak{m}_K$ is a field and it is called the {\em residue field} of $(K,\lvert\blank\rvert)$.

We denote by $K^{*}$ the multiplicative group of $K$. The set $\lvert K^{*} \rvert:=\{\lvert x \rvert\colon x\in K^{*}\}$ is a subgroup of the multiplicative group $\mathbb{R}_{+}$.
If $\lvert K^{*} \rvert=\{1\}$, the valuation $\lvert\blank\rvert$ is called {\em trivial}. If there exists $t:=\max{\lvert K^{*} \rvert\cap(0,1)}$, then $\lvert K^{*} \rvert=\{t^n \colon n\in \mathbb{Z}\}$ and the valuation $\lvert\blank\rvert$ is called {\em discrete}. In the remaining case $\lvert K^{*} \rvert$ is a dense subset of $[0,\infty)$ and the valuation $\lvert\blank\rvert$ is called {\em dense}.

By Krull's theorem \cite[Theorem 14.1]{UC}, if $(K,\lvert\blank\rvert)$ is a non-Archimedean valued field and $L$ is an extension of $K$, then there exists a non-Archimedean valuation $\lvert\blank\rvert_L$ on $L$ that extends $\lvert\blank\rvert$. Moreover, if $(K,\lvert\blank\rvert)$ is complete and $L$ is an algebraic extension of $K$, then such an extension $\lvert\blank\rvert_L$ is unique \cite[Theorem 14.2]{UC}.

A field $K$ is called {\em locally finite} if any finitely generated subfield of $K$ is finite. Evidently, $K$ is locally finite if and only if it is an algebraic extension of a finite field. Locally finite fields are precisely those that admit only the trivial valuation \cite[Corollary 14.3]{UC}.

By Ostrowski's theorem \cite[Theorem 10.1]{UC}, every nontrivial non-Archimedean valuation on $\mathbb{Q}$ is equivalent to the {\em $p$-adic valuation $\lvert\blank\rvert_p$}, for some prime number $p$, defined as follows:
\begin{equation}\label{p_adic_val}
    \lvert 0 \rvert_p:= 0, \; \left\lvert p^k \frac{m}{n}\right \rvert_p:=p^{-k} \quad \text{for } k,m,n\in\mathbb{Z}, \quad p\nmid m, p\nmid n.
\end{equation}
The completion of $(\mathbb{Q}, \lvert\blank\rvert_p)$ is called the {\em field of $p$-adic numbers} \cite[Definition 1.2.7]{PG} and denoted by $(\mathbb{Q}_p, \lvert\blank\rvert_p)$. It is a standard example of a non-Archimedean locally compact field.

Let $(K,\lvert\blank\rvert)$ be a non-Archimedean valued field. We have the following dichotomy. Assume that $K$ and its residue field $k$ have different characteristics, i.e., $\ch{K}=0$ and $\ch{k}=p$ for some prime number $p$. Then the valuation $\lvert \blank\rvert$ restricted to the prime subfield $\mathbb{Q}\subseteq K$ is nontrivial, so it equals $\lvert\blank\rvert_{p}^{c}$ for some $c>0$, where $\lvert \blank\rvert_p$ is the $p$-adic valuation defined by \eqref{p_adic_val}.
On the other hand, if $\ch{K}=\ch{k}$, then the valuation $\lvert \blank\rvert$ restricted to the prime subfield of $K$ is trivial.

\medskip

Let $X$ be a linear space over a non-Archimedean valued field $(K,\lvert\blank\rvert)$. A function $\lVert \blank \rVert \colon X\to [0, \infty)$ is called a {\em non-Archimedean norm} (briefly: norm) on $X$ if for all $x,y \in X$ and $\alpha\in K$ we have the following conditions:

\begin{enumerate*}[label=(\arabic*)]
   \item $\lVert x \rVert =0 \Leftrightarrow x=0$,
   \item $\lVert \alpha x \rVert = \lvert \alpha \rvert \lVert x \rVert $,
   \item $\lVert x+y \rVert  \leq \max\{\lVert x \rVert ,\lVert y \rVert \}$.
\end{enumerate*}\\
The pair $(X,\lVert \blank \rVert )$ is then called a {\em non-Archimedean normed space} over the non-Archimedean valued field $(K,\lvert\blank\rvert)$.
The norm $\lVert \blank \rVert $ induces an ultrametric $d$ on $X$, namely $d(x,y):=\lVert x-y \rVert $.
We say that two norms $\lVert\blank\rVert_1$, $\lVert\blank\rVert_2$ on $X$ are {\em equivalent} if there exist $C_1, C_2 >0$ such that
$C_1 \lVert x\rVert_1 \le \lVert x\rVert_2 \le C_2 \lVert x\rVert_1$ for all $x\in X$. If $(K,\lvert\blank\rvert)$ is complete and $X$ is finite-dimensional, then any two norms on $X$ are equivalent \cite[Theorem 13.3]{UC}.

The norm $\lVert x \rVert_{\infty} :=\max \{\lvert x_1 \rvert, \dots, \lvert x_n \rvert\}$, for $x=(x_1,\dots,x_n)\in K^n$, is called the {\em maximum norm}. It is the standard norm on $K^n$. An example of a norm on $K^n$ equivalent to $\lVert\blank\rVert_{\infty}$ is the {\em weighted norm} $\lVert\blank\rVert_{w}$ for a vector $w=(w_1,\dots,w_n)$ of positive real numbers, defined by {$\lVert x\rVert_{w}:=\max\{w_1 \lvert x_1 \rvert, \dots, w_n \lvert x_n \rvert\}$} for $x=(x_1,\dots,x_n)\in K^n$.

% It is easy to check that a norm $\Vert \blank\rVert$ on $K^n$ is equivalent to $\lVert\blank\rVert_{\infty}$ if and only if there exists $t\in(0,1]$ such that $(K^n,\lVert\blank\rVert)$ admits a {\em $t$-orthogonal basis}, i.e., a basis $\{b_1,\dots,b_n\}$ satisfying $\lVert \sum_{i=1}^{n} \alpha_i b_i\rVert \ge t \max_{1\le i \le n}\lVert\alpha_i b_i\rVert$ for all $\alpha_1,\dots,\alpha_n \in K$.
% It is also easily seen that if $v_1,\dots,v_n\in K^n\setminus \{0\}$ and $\frac{\lVert v_i\rVert}{\lVert v_j \rVert}\not\in \lvert K^{*}\rvert$ for $i\ne j$, then $\{v_1,\dots,v_n\}$ is a $1$-orthogonal basis of $(K^n,\lVert \blank\rVert)$.

\medskip

Let $K$ be a field and $n\in \mathbb{N}$.
Recall that an {\em affine mapping} $g\colon K^n \to K^n$ is a mapping of the form
\[
    g(x)=L(g)x+\tau(g), \quad x\in K^n,
\]
where $L(g)\in \M(n,K)$ is a linear mapping and $\tau(g) \in K^n$ is a fixed vector; $L(g)$ and $\tau(g)$ are called the {\em linear} and {\em translation} part of $g$, respectively. Notice that $g$ is invertible if and only if $L(g)\in\GL(n,K)$. The group of all invertible affine transformations of $K^n$ is called the ($n$-th) {\em general affine group} over $K$ and denoted by $\GA(n,K)$. It is isomorphic to the semidirect product $\GL(n,K)\ltimes K^n$ with respect to the natural action of $\GL(n,K)$ on the additive group $K^n$; under this isomorphism we can write $g=(L(g),\tau(g))$. The mapping $g\mapsto L(g)$ is a group homomorphism from $\GA(n,K)$ onto $\GL(n,K)$. If $H_1$ is a subgroup of $\GL(n,K)$ and $H_2$ is an $H_1$-invariant subgroup of $K^n$, then $H_1\ltimes H_2$ is a subgroup of $\GA(n,K)$. We will need some subgroups of this form later. For instance, if $R\subseteq K$ is a subring containing $1$, we put $\SA(n,R):= \SL(n,R) \ltimes R^n$, where
$\SL(n,R)=\{A\in \GL(n,R)\colon \det{A}=1\}$.

\begin{deff}\label{congruence_subgroup}
    Let $(K,\lvert\blank\rvert)$ be a non-Archimedean valued field, $\varepsilon\in (0,1]$, $n\in\mathbb{N}$ and $H\le \GL(n,D_K)$. The {\em $\varepsilon$-principal congruence subgroup} of $H$ is the subgroup $H\cap \GL(n,D_K,\varepsilon)$, where
    \begin{equation}\label{eq:GL(n,D)}
        \GL(n,D_K,\varepsilon):=\{[a_{i,j}]\in \GL(n,D_K)\colon \lvert a_{i,j}-\delta_{i,j}\rvert \le \varepsilon \; \textrm{for } 1\le i,j\le n\}
    \end{equation}
    and $[\delta_{i,j}]=I_n$ denotes the identity matrix.
    Any subgroup of $H$ containing $H\cap \GL(n,D_K,\varepsilon)$ for some $\varepsilon\in (0,1]$ is called a {\em congruence subgroup} of $H$.
\end{deff}

Notice that $\GL(n,D_K,\varepsilon)$ is the kernel of the group homomorphism
$\GL(n,D_K)\to \GL(n,D_K/B[0,\varepsilon])$ induced by the quotient homomorphism of rings $D_K \to D_K /B[0,\varepsilon]$. See \cite[Chapter 6]{Lub} for more about congruence subgroups and their applications.

For a subring $R\subseteq D_K$ containing $1$, we will write $\GL(n,R,\varepsilon)$ and $\SL(n,R,\varepsilon)$ for the $\varepsilon$-principal congruence subgroup of $\GL(n,R)$ and $\SL(n,R)$, respectively. We us also define two important classes of subgroups of $\GA(n,R)$, namely:
\begin{gather*}
    \GA(n,R,\varepsilon):=\GL(n,R,\varepsilon) \ltimes (R^n \cap (B[0,\varepsilon])^n),\\
    \SA(n,R,\varepsilon):=\SL(n,R,\varepsilon) \ltimes (R^n \cap (B[0,\varepsilon])^n).
\end{gather*}
The semidirect products above are well defined since $(R^n \cap (B[0,\varepsilon])^n)$ is a $\GL(n,R,\varepsilon)$-invariant subgroup of $K^n$. Indeed, we only need to check that $Ax \in (B[0,\varepsilon])^n$ for any 
$A\in \GL(n,D_K,\varepsilon)$ and $x=(x_1\dots,x_n)\in (B[0,\varepsilon])^n$. It is sufficient to show that $y:=Ax -x\in (B[0,\varepsilon])^n$. By \eqref{eq:GL(n,D)}, we have
$\lvert y_i\rvert \le \max_{1\le j\le n} \lvert a_{i,j}-\delta_{i,j}\rvert \cdot \lvert x_j \rvert \le \varepsilon^2 \le \varepsilon$ for any $1\le i\le n$, as desired.

\subsection{Paradoxical decompositions}

We say that a nonempty subset $E$ of a $G$-set $X$ (i.e., a set $X$ with an action of a group $G$ on $X$) is {\em $G$-paradoxical using $r$ pieces} \cite[Definition 5.1]{W} if for some positive integers $m$, $n$ with $m+n=r$ there exist pairwise disjoint subsets $A_1,\dots,A_m,B_1,\dots,B_n$ of $E$ and elements $g_1,\dots,g_m,h_1,\dots,h_n\in G$ such that
\begin{equation}\label{rparwzor}
    E=A_1 \sqcup \dots \sqcup A_m \sqcup B_1 \sqcup \dots \sqcup B_n =\bigsqcup_{i=1}^{m} g_i(A_i) = \bigsqcup_{j=1}^{n} h_j(B_j),
\end{equation}
where the symbols $\sqcup$ and $\bigsqcup$ indicate that the components of the respective unions are pairwise disjoint.
In such a case the decomposition \eqref{rparwzor} of $E$ is called {\em a $G$-paradoxical decomposition}.
In general, we do not require $E$ to be $G$-invariant; consider for example the classical Banach--Tarski paradox, where the unit ball in $\mathbb{R}^3$ is not invariant under the isometry group of $\mathbb{R}^3$ (the invariance is violated by translations).

We say that an action of a group $G$ on a set $X$ is {\em locally commutative} if $g(x)=h(x)=x$ implies $gh=hg$ for any $g,h\in G$ and $x\in X$ (in other words, the stabilizer of each $x\in X$ is an Abelian subgroup of $G$).

It is known that a $G$-set $X$ is $G$-paradoxical using four pieces if and only if $G$ contains a free subgroup of rank two $F_2$ whose action on $X$ is locally commutative. More precisely, $X$ has a decomposition
$X=A_1 \sqcup A_2 \sqcup A_3 \sqcup A_4 = A_1 \sqcup a(A_2) = A_3 \sqcup b(A_4)$ for $a,b \in G$ if and only if $\{a,b\}$ is a basis of a free subgroup of rank two $F_2\le G$ whose action on $X$ is locally commutative \cite[Theorems 5.5 and 5.8]{W}.

\section{Results}

\subsection{Linear and affine isometries of normed spaces $(K^n,\lVert\blank\rVert)$}

Let $(K,\lvert\blank\rvert)$ be a non-Archimedean valued field and $\lVert\blank\rVert$ a non-Archimedean norm on $K^n$ for some $n\in \mathbb{N}$. We denote by $\IL(K^n,\lVert\blank\rVert)$ and $\IA(K^n,\lVert\blank\rVert)$ the groups of all linear and, respectively, affine isometries of the normed space $(K^n,\lVert\blank\rVert)$.
Clearly, we have $\IA(K^n,\lVert\blank\rVert) = \IL(K^n,\lVert\blank\rVert) \ltimes K^n$.

We will prove that if a norm $\lVert\blank\rVert$ on $K^n$ is equivalent to the norm $\lVert\blank\rVert_{\infty}$, then the intersection
$\IL(K^n,\lVert\blank\rVert)\cap \GL(n,D_K)$ is a congruence subgroup of $\GL(n,D_K)$.

\begin{prop}\label{IL_are_congruence_groups}
    Let $(K,\lvert\blank\rvert)$ be a non-Archimedean valued field and $n\in\mathbb{N}$. If a norm $\lVert\blank\rVert$ on $K^n$ is equivalent to the norm $\lVert\blank\rVert_{\infty}$, then there exists $\varepsilon\in (0,1]$ such that
    $\GL(n,D_K,\varepsilon)\le \IL(K^n,\lVert\blank\rVert)$ and $\GA(n,D_K,\varepsilon)\le \IA(K^n,\lVert\blank\rVert)$.
\end{prop}

\begin{proof}
    Let $C_1,C_2>0$ be such that $C_1 \lVert x\rVert_{\infty} \le \lVert x\rVert \le C_2 \lVert x\rVert_{\infty}$ for all $x\in K^n$. Let us put $\varepsilon:=\frac{C_1}{C_2}$.
    Assume that $A\in\GL(n,D_K,\varepsilon)$ and denote $B=[b_{i,j}]:=A-I_n$. Then, for all $x=(x_1,\dots,x_n)\in K^n$, we have
    \begin{equation*}
        \lVert Bx\rVert_{\infty}=\max_{1\le i\le n}\left\lvert \sum_{j=1}^{n} b_{i,j}x_j\right\rvert \le \max_{1\le i,j\le n} \lvert b_{i,j} \rvert \lvert x_j \rvert\le \varepsilon \lVert x\rVert_{\infty}.
    \end{equation*}
    Hence, $\lVert Ax -x\rVert =\lVert Bx\rVert \le C_2 \lVert Bx\rVert_{\infty} \le C_{2}\,\varepsilon \lVert x\rVert_{\infty}\le C_1 \lVert x\rVert_{\infty}\le \lVert x\rVert$, from which $\lVert Ax\rVert\le \lVert x\rVert$ follows. Since $A^{-1}\in \GL(n,D_{K},\varepsilon)$, we similarly obtain $\lVert A^{-1}x\rVert\le \lVert x\rVert$, so $A\in \IL(K^n,\lVert\blank\rVert)$.

    As we have shown the containment $\GL(n,D_K,\varepsilon)\le \IL(K^n,\lVert\blank\rVert)$, it becomes obvious that $\GA(n,D_K,\varepsilon)\le \IA(K^n,\lVert\blank\rVert)$.
\end{proof}

\begin{wn}\label{cor:complete}
    If $(K,\lvert\blank\rvert)$ is a complete non-Archimedean valued field, $n\in\mathbb{N}$ and $\lVert \blank \rVert$ is a norm on $K^n$, then there exists $\varepsilon\in (0,1]$ such that
    $\GL(n,D_K,\varepsilon)\le \IL(K^n,\lVert\blank\rVert)$ and $\GA(n,D_K,\varepsilon)\le \IA(K^n,\lVert\blank\rVert)$.
\end{wn}

% Let us formulate a simple lemma.

% \begin{lem}
%     Let $(K,\lvert\blank\rvert)$ be a non-Archimedean valued field, $n\in \mathbb N$ and $w=(w_1,\dots,w_n)\in\mathbb{R}_{+}^{n}$. There exists $\varepsilon_0 \in (0,1]$ such that
%     $\SL(n,D_K,\varepsilon)\le \IL_{w}(n,K)$ and
%     $\SA(n,D_K,\varepsilon)\le \IA_{w}(n,K)$
%     for all $\varepsilon \in (0,\varepsilon_0]$.
% \end{lem}

% \begin{proof}
%     It is sufficient to take $\varepsilon_0:=\min_{1\le i,j\le n} \frac{w_j}{w_i}$. Since $\lvert a_{i,i}-1\rvert\le \varepsilon_0$ implies $\lvert a_{i,i}\rvert \le 1$ for $a_{i,i}\in D_K$, all the inequalities from the condition (3) of Theorem \ref{isoweight} are satisfied for any $A\in \SL(n,D_K,\varepsilon)$ and $\varepsilon \in (0,\varepsilon_0)$.
% \end{proof}

\subsection{Paradoxical decompositions of normed spaces $(K^n, \lVert\blank\rVert)$}

Assume that $(K,\lvert\blank\rvert)$ is a non-Archimedean nontrivially valued field. We are going to prove that, for any $n\ge 2$ and $\varepsilon \in (0,1]$, the group $\SA(n,D_K,\varepsilon)$ contains a free subgroup of rank two $F_2$ acting locally commutatively on $K^n$. That will yield a paradoxical decomposition of $K^n$ using four pieces.

\medskip

It turns out that, after a suitable embedding $\iota$ of $\GA(2,K)$ into $\GA(n,K)$, a locally commutative action of a subgroup $G \le \GA(2,K)$ on $K^2$ gives rise to a locally commutative action of $\iota(G)\le \GA(n,K)$ on $K^n$ for $n>2$. Under the lemma below, we can thus focus on the case $n=2$.

\begin{lem}\label{embedding}
    Let $(K,\lvert\blank\rvert)$ be a non-Archimedean valued field, $n>2$ and $1\le i< n$. For $g=(A,\tau)\in \GA(2,K)$, define $\iota(g):=(A',\tau')\in \GA(n,K)$ as follows. The matrix $A'$ is given by
    \[
    A':=
    \begin{pmatrix}
       1 & 0 &  & &  0 \\
       0 & \ddots &  \ddots & & \\
        & \ddots & A & \ddots &   \\
        & & \ddots & \ddots & 0 \\
      0 & & & 0 & 1   
    \end{pmatrix}
    ,\]
    where the block $A$ occupies the rows and columns labeled $i$ and $i+1$, and the vector $\tau'$ is defined by
    $\tau'_{i}:=\tau_1$, $\tau'_{i+1}=\tau_2$, $\tau'_{s}=0$ for $s\not\in\{i,i+1\}$.

    Then the map $\iota\colon \GA(2,K)\to \GA(n,K)$ is a group embedding and $\iota(\SA(2,D_K,\varepsilon))\le \iota(\SA(n,D_K,\varepsilon))$ for any $\varepsilon \in (0,1]$. Moreover, if a subgroup $G\le \GA(2,K)$ acts locally commutatively on $K^2$, then so does $\iota(G)$ on $K^n$.
\end{lem}

\begin{proof}
    It is not hard to check that $\iota$ is an embedding and $\iota(\SA(2,D_K,\varepsilon))\le \iota(\SA(n,D_K,\varepsilon))$ for any $\varepsilon \in (0,1]$.
    Assume that $\iota(g)$, $\iota(h)$ have a common fixed point $x\in K^n$ for some $g,h\in G$. By the definition of $\iota$, $(x_i,x_{i+1})$ is then a common fixed point of $g$ and $h$. Hence, $gh=hg$ and $\iota(g)\iota(h)=\iota(h)\iota(g)$.
\end{proof}

Let us begin with the case when $\ch{K}=0$ and $\ch{k}=p$ for some prime number $p$.
We will apply a result of Magnus \cite[Theorems 3 and 4]{Magnus}, where he constructed an explicit basis of a free subgroup of $\SL(2,\mathbb{Z})$ with some additional properties. Then we will use a recent result of Pham \cite[Theorem 1.1]{Pham}, which can be thought of as an affine variant of the {\em uniform Tits alternative} with the additional property of local commutativity (see \cite[Theorems 1 and 2]{Tits} for the classical Tits alternative).

\begin{prop}\label{different_characteristics}
    Let $(K,\lvert\blank\rvert)$ be a non-Archimedean valued field with the residue field $k$. Assume that $\ch{K}\ne \ch{k}$. Then, for any $n\ge 2$ and $\varepsilon\in (0,1]$, there exists a free subgroup of rank two $F_2 \le \SA(n,\mathbb{Z},\varepsilon)$ acting locally commutatively on $K^n$.
\end{prop}

\begin{proof}
    Assume first that $n=2$. Let $\varepsilon \in (0,1]$, $p=\ch{k}$ and $s, t\in \mathbb{N}$, $s< t$ be such that $\lvert p^s\rvert\le \varepsilon$. Let us put
    \begin{equation*}
        A_m:=\left(
        \begin{array}{cc}
                4m^2 + 1 & 2m \\
                2m       & 1
            \end{array}
        \right),\quad
        m\in \mathbb{N}.
    \end{equation*}
    It follows from \cite[Theorem 3 and 4]{Magnus} that $\{A_{p^s}, A_{p^t}\}$ is a basis of a free subgroup $F'_2$ of $\SL(2,\mathbb{Z})$. Moreover, every element of $F'_2\setminus\{I_2\}$ is {\em nonparabolic}, i.e., its trace is different from $\pm 2$. By our choice of $s$, we have $F'_2 \le \SL(2,\mathbb{Z},\varepsilon)$.

    Let $\tau:=(2p^s,0)$ and $g:=(A_{p^s},\tau)$, $h:=(A_{p^t},\tau) \in \SA(2,\mathbb{Z},\varepsilon)$. Obviously, $\{g,h\}$ is a basis of a free subgroup $\Gamma \le \SA(2,\mathbb{Z},\varepsilon)$. Furthermore, if $\gamma=(L(\gamma), \tau(\gamma))\in \Gamma \setminus\{1_{\Gamma}\}$, then $L(\gamma)\in F'_2\setminus\{I_2\}$. Notice that the characteristic polynomial of $L(\gamma)$ is $\lambda^2 - \tr{L(\gamma)}+1$. Hence,
    $\det(I_2-L(\gamma))=2-\tr L(\gamma)\ne 0$ and
    $\Fix(\gamma):=(I_2-L(\gamma))^{-1}\tau(\gamma)\in \mathbb{Q}^2$ is the unique fixed point of $\gamma$ in $K^2$. We have $\Fix(g)=(0,-1)\ne (0,-p^{s-t})=\Fix(h)$. It follows that $\Gamma$ does not have a global fixed point in $\mathbb{Q}^2$. Let us put $S:=\{g,h,g^{-1},h^{-1},1_{\Gamma}\}$.

    Since the group $\Gamma$, generated by $S$, is a non-Abelian free group, it is not virtually solvable. By an application of \cite[Theorem 1.1]{Pham}, there exists $N\in\mathbb{N}$ such that the product $S^N \subseteq \Gamma$ contains two elements $a,b$ forming a basis of a free group of rank two $F_2$ whose action on $\mathbb{Q}^2$ is locally commutative. Clearly, $F_2\le \Gamma\le  \SA(2,\mathbb{Z},\varepsilon)$. It is easy to see that any $x\in K^2 \setminus \mathbb{Q}^2$ has the trivial stabilizer in $\Gamma$, so the action of $F_2$ is locally commutative on $K^2$ as well.

    We get the claim for $n>2$ applying Lemma \ref{embedding}.
\end{proof}

To obtain a similar result for the case when $\ch{K}=\ch{k}$, we will need another technique. It is borrowed from the reasoning in the proof of \cite[Proposition C.4]{Tao}. However, in the non-Archimedean case many calculations are easier and, provided that some simple conditions are fulfilled, we are able to find explicit generators of $F_2$ of a form that will be useful later in the proof of Proposition \ref{equal_characteristics}.

It is worth noting that the paper \cite{Tao} deals with an expansion property of random Cayley graphs of certain finite groups. Appendix C therein is auxiliary in nature and serves as an intermediate step in the proof of the main theorem in an exceptional case, namely for the symplectic group of order $4$ over a finite field of characteristic $3$.

Pham formulated the following abstract lemma providing locally commutative actions of $F_2$ under certain assumptions. It is inspired from \cite[Proposition C.4]{Tao} and will be called the {\em ping-pong lemma}.

\begin{lem}[{\cite[Lemma 3.2]{Pham}}]\label{Pham's_ping-pong}
    Let $X$ be a $G$-set and $S=\{a,a^{-1},b,b^{-1}\}\subseteq G$. Assume that, for each $s\in S$, we are given sets $U_{s}^{+}, U_{s}^{-}\subseteq X$ with the following properties:
    \begin{align}
        s(X\setminus U_{s}^{-})               & \subseteq  U_{s}^{+} & \text{for }  & s\in S, \label{eq:Pham1}                           \\
        U_{x}^{+}\cap U_{y}^{-}               & =\emptyset           & \text {for } & x,y\in S, \;y\ne x^{-1},\label{eq:Pham2}           \\
        U_{x}^{-}\cap U_{y}^{-}\cap U_{z}^{-} & =\emptyset           & \text {for } & x,y,z\in S, \;\text{all distinct}.\label{eq:Pham3}
    \end{align}
    Further, assume that there exist a function $f\colon X\to \mathbb{R}$ such that
    \begin{equation}\label{eq:Pham4}
        f(sx)>f(x), \quad \text{for } x\in X\setminus U_{s}^{-},\quad s\in S.
    \end{equation}
    Then, $\{a,b\}$ is a basis of a non-Abelian free group $F_2\le G$ whose action on $X$ is locally commutative.
\end{lem}

Let us prove the following result, which makes use of the ping-pong lemma.

\begin{lem}\label{auxiliary}
    Let $(K,\lvert\blank\rvert)$ be a non-Archimedean valued field. Assume that $\alpha, \beta, \gamma, \delta, \tau_1, \tau_2, \lambda\in K$ satisfy the following conditions:
    \begin{equation}\label{aux_conditions}
        \begin{gathered}
            \lvert \alpha\rvert = \lvert\beta\rvert = \lvert\gamma\rvert = \lvert\delta\rvert = \lvert\tau_1\rvert = \lvert\tau_2\rvert =
            \lvert \beta \tau_2 - \delta \tau_1 \rvert =
            \lvert \gamma \tau_1 - \alpha \tau_2\rvert, \\
            \lvert \lambda \rvert > 1,
            \qquad \alpha \delta - \beta \gamma =1.
        \end{gathered}
    \end{equation}
    Let also
    $a:=\Big( \begin{array}{cc}
                \lambda & 0            \\
                0       & \lambda^{-1}
            \end{array}
        \Big)$,
    $L:=\Big( \begin{array}{cc}
                \alpha & \beta  \\
                \gamma & \delta
            \end{array}
        \Big)$, $\tau:=(\tau_1,\tau_2)$, $h:=(L,\tau)$ and $b:=hah^{-1}$.
    Then $\{a,b\}$ is a basis of a free subgroup $F_2$ of $\SA(2,K)$ whose action on $K^2$ is locally commutative.
\end{lem}

\begin{proof}
    Let us define the sets
    \begin{equation*}
        \begin{split}
            U_{a}^{-}&:=\{(x,y)\in K^2 \colon \lVert(x,y)\rVert_{\infty} < 1 \text{ or } \lvert y\rvert > \lvert x\rvert\},\\
            U_{a^{-1}}^{-}&:=\{(x,y)\in K^2 \colon \lVert(x,y)\rVert_{\infty} < 1 \text{ or } \lvert x\rvert > \lvert y\rvert\},\\
            U_{a}^{+}&:=\{(x,y)\in K^2 \colon \lvert x\rvert > \lVert ( y, 1) \rVert_{\infty}\},\\
            U_{a^{-1}}^{+}&:=\{(x,y)\in K^2 \colon \lvert y\rvert > \lVert (x, 1) \rVert_{\infty}\}.
        \end{split}
    \end{equation*}
    Moreover, let $U_{b}^{-}:=h(U_{a}^{-})$, $U_{b^{-1}}^{-}:=h(U_{a^{-1}}^{-})$, $U_{b}^{+}:=h(U_{a}^{+})$, $U_{b^{-1}}^{+}:=h(U_{a^{-1}}^{+})$.

    We will show that the assumptions of Lemma \ref{Pham's_ping-pong} are fulfilled for $X=K^2$, $G=\SA(2,K)$, the sets $U_{s}^{-}$, $U_{s}^{+}$, $s\in S:=\{a,b,a^{-1},b^{-1}\}$, and the function $f=\lVert \blank \rVert_{\infty}\colon K^2 \to \mathbb{R}$.

    Assume that $(x,y)\in K^2\setminus U_{a}^{-}$. Since $\lvert x\rvert \ge \max\{\lvert y\rvert, 1\}$, we have $\lvert \lambda x\rvert > \max\{\lvert \lambda^{-1} y\rvert, 1\}$. Thus,
    $a(x,y)=(\lambda x, \lambda^{-1}y)\in U_{a}^{+}$
    and 
    \begin{equation}\label{eq:aux_1}
    \lVert a(x,y)\rVert_{\infty} = \lvert \lambda x\rvert = \lvert \lambda \rvert \cdot \lVert (x,y)\rVert_{\infty}>\lVert (x,y)\rVert_{\infty} \; \text{for } (x,y)\in K^2\setminus U_{a}^{-}.
    \end{equation}
    Similarly, if $(x,y)\in K^2\setminus U_{a^{-1}}^{-}$, then
    $a^{-1}(x,y)\in U_{a^{-1}}^{+}$ and $\lVert a^{-1}(x,y)\rVert_{\infty} = \lvert \lambda \rvert \cdot \lVert (x,y)\rVert_{\infty}>\lVert (x,y)\rVert_{\infty}$. Hence, we have proved the conditions \eqref{eq:Pham1} and \eqref{eq:Pham4} for $s\in \{a, a^{-1}\}$.
    
    The inclusions \eqref{eq:Pham1} for $s\in \{b,b^{-1}\}$ follow immediately from those already proved.

    The inverse of $h$ in $\SA(2,K)$ has the form $h^{-1}=(L^{-1},\tau')$, where $L^{-1}=\Big( \begin{array}{cc}
                \delta  & -\beta \\
                -\gamma & \alpha
            \end{array}
        \Big)$ and $\tau'=(\tau'_1,\tau'_2)$ for
    $\tau'_1:=\beta \tau_2 - \delta \tau_1$ and $\tau'_2:=\gamma \tau_1 - \alpha \tau_2$. The assumptions \eqref{aux_conditions} imply that all the entries of the linear part of $h^{-1}$ and the coordinates of its translation part have equal valuations, exactly as in the case of $h$.

    It is obvious that $U_{s}^{+}\cap U_{s}^{-}=\emptyset$ for $s\in S$.
    
    The remaining conditions in \eqref{eq:Pham2} follow from the inclusion
    \[h(U_{a}^{+}\cup U_{a^{-1}}^{+})\cup h^{-1}(U_{a}^{+}\cup U_{a^{-1}}^{+}) \subseteq (K^2 \setminus U_{a}^{-}) \cap (K^2 \setminus U_{a^{-1}}^{-}),\] where the latter set equals
    $\big\{(x,y)\in K^2 \colon \lvert x\rvert=\lvert y\rvert \ge 1\big\}$. Assume that $(x,y)\in U_{a}^{+}$ and let $(x',y'):=h(x,y)$.
    Since $\lvert x\rvert > \max\{\lvert y\rvert, 1\}$, by \eqref{aux_conditions} we have
    $\lvert \alpha x \rvert > \max\{\lvert \beta y \rvert, \lvert \tau_1 \rvert\}$
    and $\lvert \gamma x \rvert> \max\{\lvert\delta y\rvert, \lvert \tau_2\rvert \}$; so
    $\lvert x'\rvert = \lvert \alpha x + \beta y + \tau_1 \rvert=\lvert \alpha x\rvert$ and similarly $\lvert y'\rvert =\lvert\gamma x+\delta y + \tau_2 \rvert= \lvert \gamma x\rvert$. Notice that from \eqref{aux_conditions} it follows that $\lvert\alpha \rvert \ge 1$. Thus, $\lvert x'\rvert = \lvert y'\rvert =\lvert \alpha x\rvert >1$, so $(x',y')\in (K^2 \setminus U_{a}^{-}) \cap (K^2 \setminus U_{a^{-1}}^{-})$. If $(x,y)\in U_{a^{-1}}^{+}$, we just change the roles of $x$ and $y$, obtaining $h(x,y)\in (K^2 \setminus U_{a}^{-}) \cap (K^2 \setminus U_{a^{-1}}^{-})$. Simultaneously, we have shown that
    \begin{equation}\label{eq:aux_2}
    \lVert h(x,y)\rVert_{\infty}=\lvert \alpha \rvert \cdot \lVert (x,y)\rVert_{\infty} \; \text{for } (x,y)\in U_{a}^{+}\cup U_{a^{-1}}^{+}.
    \end{equation} 

    From the remark about the form of $h^{-1}$, it becomes evident that $h^{-1}(U_{a}^{+}\cup U_{a^{-1}}^{+}) \subseteq (K^2 \setminus U_{a}^{-}) \cap (K^2 \setminus U_{a^{-1}}^{-})$ as well.

    The conditions \eqref{eq:Pham3} follow from the inclusion
    \[
        h(U_{a}^{-}\cap U_{a^{-1}}^{-})\cup h^{-1}(U_{a}^{-}\cap U_{a^{-1}}^{-}) \subseteq (K^2 \setminus U_{a}^{-}) \cap (K^2 \setminus U_{a^{-1}}^{-}).
    \]
    Notice that $U_{a}^{-}\cap U_{a^{-1}}^{-}$ is the open ball $B(0,1)$ in $(K^2,\lVert\blank\rVert_{\infty})$.
    Assume that $(x,y)\in B(0,1)$ and let $(x',y'):=h(x,y)$. Since
    $\lvert \alpha x + \beta y \rvert \le \lvert \alpha \rvert \cdot \lVert (x,y)\rVert_{\infty} <  \lvert \tau_1\rvert$, we have $\lvert x' \rvert = \lvert \tau_1 \rvert$ and similarly $\lvert y' \rvert = \lvert \tau_2\rvert$. Thus, $\lvert x'\rvert = \lvert y'\rvert =\lvert \tau_1 \rvert \ge 1$, so $h(U_{a}^{-}\cap U_{a^{-1}}^{-}) \subseteq (K^2 \setminus U_{a}^{-}) \cap (K^2 \setminus U_{a^{-1}}^{-})$. The proof for $h^{-1}$ is analogous.

    It remains to show \eqref{eq:Pham4} for $s\in \{b,b^{-1}\}$. First, let us observe that
    \begin{equation}\label{eq:aux_3}
    \lVert h(x,y)\rVert_{\infty} \le \lvert \alpha \rvert \cdot \lVert (x,y) \rVert_{\infty} \; \text{for } (x,y)\in (K^2\setminus U_{a}^{-}) \cup (K^2\setminus U_{a^{-1}}^{-}).
    \end{equation}
    Indeed, if for example $(x,y)\in K^2\setminus U_{a}^{-}$ and $(x',y'):=h(x,y)$, then $\lvert x'\rvert = \lvert \alpha x + \beta y + \tau_1 \rvert \le \lvert \alpha x \rvert = \lvert \alpha \rvert  \cdot \lVert (x,y) \rVert_{\infty}$ and similarly $\lvert y'\rvert \le  \lvert \alpha \rvert \cdot \lVert (x,y) \rVert_{\infty}$. Assume now that $(x,y)\in K^2 \setminus U_{b}^{-}$. Then $h^{-1}(x,y)\in K^2 \setminus U_{a}^{-}$, so $ah^{-1}(x,y)\in U_{a}^{+}$. We thus obtain $\lVert b(x,y)\rVert_{\infty} = \lVert hah^{-1}(x,y)\rVert_{\infty}=\lvert \alpha \rvert \cdot \lVert ah^{-1}(x,y)\rVert_{\infty}$ by \eqref{eq:aux_2}.
    Then, from \eqref{eq:aux_1} we get
    $\lVert ah^{-1}(x,y)\rVert_{\infty} = \lvert \lambda \rvert \cdot \lVert h^{-1}(x,y)\rVert_{\infty}$.
    Since $\lVert (x,y)\rVert_{\infty}\le \lvert \alpha\rvert \cdot \lVert h^{-1}(x,y)\rVert_{\infty}$ by \eqref{eq:aux_3}, eventually we obtain
    \[\lVert b(x,y)\rVert_{\infty}= \lvert\alpha\rvert\cdot \lvert\lambda\rvert \cdot \lVert h^{-1}(x,y)\rVert_{\infty}\ge   \lvert\lambda\rvert\lVert (x,y)\rVert_{\infty}>\lVert (x,y)\rVert_{\infty}\] 
    if $(x,y)\ne (0,0)$. If $(x,y)=(0,0)$, then $\lVert b(x,y)\rVert_{\infty}=\lvert \alpha \rvert \cdot \lvert \lambda \rvert \cdot \lVert \tau'\rVert_{\infty}=\lvert \alpha\rvert^2 \cdot \lvert \lambda\rvert >0=\lVert (x,y)\rVert_{\infty}$. The proof of \eqref{eq:Pham4} for $s=b^{-1}$ is similar.
\end{proof}

\begin{prop}\label{equal_characteristics}
    Let $(K,\lvert\blank\rvert)$ be a non-Archimedean nontrivially valued field with the residue field $k$. Assume that $\ch{K}=\ch{k}$. Then, for any $n\ge 2$ and $\varepsilon\in (0,1]$, there exists a free subgroup of rank two $F_2 \le \SA(n,D_{K},\varepsilon)$ acting locally commutatively on $K^n$.
\end{prop}

\begin{proof}
    Assume first that $n=2$. Let $\varepsilon\in (0,1]$ and denote by $K_0$ the prime subfield of $K$. Since $\ch{K}=\ch{k}$, the valuation $\lvert \blank \rvert$ is trivial on $K_0$. Let $t\in K$ be any element satisfying $0<\lvert t\rvert<\varepsilon$. By \cite[Theorem 14.2]{UC}, $t$ is necessarily transcendental over $K_0$, so the subfield $K_0(t)\subseteq K$ is naturally isomorphic to the field $K_0(X)$ of rational functions of one variable over $K_0$.

    Let us define an auxiliary non-Archimedean valuation $\lvert \blank \rvert_1$ on $K_0(t)$ as follows (cf. \cite [Example 2.1]{UC}):
    \[
        \lvert 0 \rvert_1:=0, \; \left\lvert \frac{f(t)}{g(t)}\right\rvert_1:= 2^{\deg{f}-\deg{g}} \quad \text{for } f,g\in K_0[X], \; g\ne 0,
    \]
    and let us extend it (by Krull's theorem \cite[Theorem 14.1]{UC}) to a non-Archimedean valuation on $K$, also denoted by $\lvert \blank\rvert_1$.

    Applying Lemma \ref{auxiliary} to the valued field $(K,\lvert \blank \rvert_1)$, we conclude that
    $a=\Big( \begin{array}{cc}
                1+t & 0          \\
                0   & (1+t)^{-1}
            \end{array}
        \Big)$ and $b=hah^{-1}$, where
    $h=\left(\Big( \begin{array}{cc}
                1+t & t   \\
                -t  & 1-t
            \end{array}
        \Big), (-t, t)\right)$,
    form a basis of a free subgroup $F_2\le \SA(2,K)$ acting locally commutatively on $K^2$. Since $\lvert t \rvert< \varepsilon$, it follows that $a,b\in \SA(2,D_{K},\varepsilon)$, where $D_K$ is the ring of integers of $(K,\lvert \blank\rvert)$; so $F_2\le \SA(2,D_{K},\varepsilon)$.

    We get the claim for $n>2$ applying Lemma \ref{embedding}.
\end{proof}

\begin{uw}
    For a transcendental extension $K=\mathbb{Q}(t)$ of $\mathbb{Q}$, the problem of finding a free subgroup of rank two $F_2\le \SA(2,K)$ acting locally commutatively on $K^2$ was solved by Sat\^{o} \cite{Sato}. However, if we equip $K$ with a non-Archimedean valuation trivial on $\mathbb{Q}$, the group $F_2$ constructed there is not contained in $\SA(2,D_K,\varepsilon)$ for any $\varepsilon \in (0,1)$. Our approach from Proposition \ref{equal_characteristics} provides a group with this additional property, and works in the case of $\ch{K}=\ch{k}>0$ as well.
\end{uw}

Let us state a theorem that combines Propositions \ref{different_characteristics} and \ref{equal_characteristics}.

\begin{tw}\label{th:SA(n,D,e)-paradoxicality}
    Let $(K,\lvert\blank\rvert)$ be a non-Archimedean nontrivially valued field, $n\ge 2$ and $\varepsilon\in(0,1]$. Then any nonempty $\SA(n,D_{K},\varepsilon)$-invariant subset $E\subseteq K^n$ is $\SA(n,D_{K},\varepsilon)$-paradoxical using four pieces.
\end{tw}

\begin{proof}
    Let us fix $\varepsilon \in (0,1]$. Depending on whether the characteristics of $K$ and its residue field $k$ are distinct or equal, we apply Proposition \ref{different_characteristics} or \ref{equal_characteristics} to obtain a free subgroup of rank two ${F_2\le \SA(n,D_{K},\varepsilon)}$ whose action on $K^n$, and so on any $\SA(n,D_{K},\varepsilon)$-invariant subset $E\subseteq K^n$, is locally commutative. Therefore, $E$ is $\SA(n,D_{K},\varepsilon)$-paradoxical using four pieces.
\end{proof}

We are now ready to prove our main result.

\begin{tw}\label{th:main}
    Let $(K,\lvert\blank\rvert)$ be a non-Archimedean nontrivially valued field, $n\ge 2$, and let $\Vert \blank \rVert$ be a non-Archimedean norm on $K^n$ equivalent to the norm $\lVert \blank\rVert_{\infty}$. Then $K^n$, any closed ball $B[x_0,r]$, any open ball $B(x_0,r)$, and any nonempty sphere $S[x_0,r]$, where $x_0\in K^n$ and $r>0$, are paradoxical with respect to the group of affine isometries of $(K^n,\lVert\blank\rVert)$ using four pieces.
\end{tw}

\begin{proof}
    By Proposition \ref{IL_are_congruence_groups}, there exists $\varepsilon_0\in (0,1]$ such that $\SA(n,D_K,\varepsilon_0)$ is a subgroup of $\IA(K^n,\lVert\blank\rVert)$. Thus, the claim for $K^n$ immediately follows from Theorem \ref{th:SA(n,D,e)-paradoxicality}.

    Fix $x_0\in K^n$ and $r>0$. Let $C>0$ be such that $\lVert x\rVert \le C \lVert x\rVert_{\infty}$ for all $x\in K^n$. Define
    \[
    \varepsilon:=
    \begin{cases}
    \frac{1}{2} \min \big\{ \varepsilon_0, \frac{r}{C}, \frac{r}{C \lVert x_0 \rVert_{\infty}} \big\} & \text{if } x_0 \neq \mathbf{0}, \\
    \frac{1}{2} \min \big\{ \varepsilon_0, \frac{r}{C} \big\} & \text{if } x_0 = \mathbf{0}.
    \end{cases}
    \]
    We now show that $\lVert g(x_0)-x_0\rVert<r$ for any $g=(A,\tau)\in \SA(n,D_K,\varepsilon)$. Estimating as in the proof of Proposition \ref{IL_are_congruence_groups}, we get 
    \begin{gather*}
    \lVert Ax_0 - x_0\rVert \le C \lVert Ax_0 - x_0\rVert_{\infty} \le C \varepsilon \lVert x_0 \rVert_{\infty}<r,\\
    \lVert \tau\rVert \le C \lVert \tau \rVert_{\infty} \le C \varepsilon < r.
    \end{gather*}
    Hence, $\lVert g(x_0)-x_0\rVert=\lVert Ax_0-x_0+\tau\rVert < r$.

    Any $g\in \SA(n,D_K,\varepsilon)$ is a surjective isometry of $K^n$; so the images of $B[x_0,r]$, $B(x_0,r)$, $S[x_0,r]$ under $g$ are equal to $B[g(x_0),r]$, $B(g(x_0),r)$ and $S[g(x_0),r]$, respectively. Since $\lVert g(x_0)-x_0\rVert < r$, we obtain {$B[g(x_0),r]=B[x_0,r]$} and $B(g(x_0),r)=B(x_0,r)$  (by the strong triangle inequality), as well as $S[g(x_0),r]=S[x_0,r]$ (by the isosceles property). Thus, the sets under consideration are $\SA(n,D_K,\varepsilon)$-invariant and we can apply Theorem \ref{th:SA(n,D,e)-paradoxicality} for each of them.
\end{proof}

It is easy to see that every finite-dimensional normed space $(X,\Vert\blank\rVert)$ over $(K,\lvert\blank\rvert)$ is linearly isometric to $(K^n, \lVert\blank\rVert^{'})$ for some norm $\lVert\blank\rVert{'}$ on $K^n$, where $n=\dim{X}$. Thus, we get the following corollary.

\begin{wn}\label{wn:fin_dim_complete}
    Let $(X,\lVert\blank\rVert)$ be finite-dimensional normed space over a complete non-Archimedean nontrivially valued field $(K,\lvert\blank\rvert)$. If ${\dim{X}\ge 2}$, then $X$, all balls and all nonempty spheres in $X$ are paradoxical with respect to the group of affine isometries of $(X,\lVert\blank\rVert)$ using four pieces.
\end{wn}

\begin{uw}
    Assume that $(K,\lvert\blank\rvert)$ is a non-Archimedean nontrivially valued field.
    We have shown (implicitly in the proofs of Proposition \ref{different_characteristics} and \ref{equal_characteristics}) that, for any $n\ge 2$ and $\varepsilon \in (0,1]$, the group $\SL(n,D_K,\varepsilon)$ contains elements which are not scalar multiples of $I_n$. Thus, by Corollary \ref{cor:complete}, if $(K,\lvert\blank\rvert)$ is complete and $n\ge 2$, every normed space $(K^n,\lVert\blank\rVert)$ admits linear isometries different from $\lambda \, I_n$, $\lambda\in K$. The situation is quite different in the Archimedean case \cite{Jarosz}, where any (real or complex) Banach space $(X,\lVert \blank\rVert)$ has an equivalent norm $\lVert\blank\rVert^{'}$ such that any surjective linear isometry of $(X,\lVert\blank\rVert^{'})$ is of the form $\lambda \Id_{X}$ for some scalar $\lambda$ with $\lvert \lambda\rvert=1$.  
\end{uw}

Let us consider the case of a normed space $(K^n,\lVert\blank\rVert)$ over a trivially valued field $(K,\lvert \blank\rvert)$. Every closed ball $B[0,r]$, $r>0$, is then a linear subspace of $K^n$, so it is easy to see that the norm $\lVert \blank\rVert$ takes at most $n$ nonzero values.

\begin{tw}\label{trivial_val}
    Let $(K,\lvert\blank\rvert)$ be a trivially valued field that is not locally finite. Let $n\ge 2$ and $\lVert\blank\rVert$ be a non-Archimedean norm on $K^n$. If the norm $\lVert \blank \rVert$ takes less than $n$ nonzero values, then $K^n$ is $\IA(K^n,\lVert \blank \rVert)$-paradoxical using four pieces; otherwise $K^n$ is not $\IA(K^n,\lVert \blank \rVert)$-paradoxical.
\end{tw}

\begin{proof}
    Let $0< a_1 < \dots < a_s$ be the sequence of all  values of $\lVert\blank \rVert$. Put $V_0:=\{0\}$, $V_i:=B[0,a_i]$ for $1\le i\le s$, and $d_i:=\dim V_i$ for $0\le i\le s$. We have the chain $V_0 \subsetneq V_1 \subsetneq \dots \subsetneq V_s=K^n$ of linear subspaces of $K^n$. Hence, there exists a basis $B=\{b_1,\dots, b_n\}$ of $K^n$ such that $\Lin{\{b_j \colon 1\le j\le d_i\}}=V_i$ for $1\le i\le s$. Clearly, a linear automorphism $L$ of $K^n$ is an isometry of $(K^n,\lVert\blank\rVert)$ if and only if $L(V_i)\subseteq V_i$ for all $1\le i\le s$.
    
    If $s=n$, it happens precisely when the matrix of $L$ in the basis $B$ is an upper-triangular invertible matrix. Since the group of all such matrices is solvable and hence amenable \cite[Theorem 3.2]{KateJ}, so is $\IA(K^n,\lVert \blank \rVert)$, which implies that $K^n$ is not $\IA(K^n,\lVert \blank \rVert)$-paradoxical \cite[Theorem A.13]{KateJ}.

    Assume that $s<n$. There exists $1\le k\le s$ such that $d_k-d_{k-1}\ge 2$. Then $\iota(\GA(2,K))$, where $\iota$ is defined as in Lemma \ref{embedding} for $i=d_{k-1}+1$, represents (in the basis $B$), a subgroup of $\IA(K^n,\lVert \blank\rVert)$. It remains to show that $\GA(2,K)$ contains a free subgroup of rank two $F_2$ acting locally commutatively on $K^2$. Indeed, since $K$ is not locally finite, it admits a nontrivial non-Archimedean valuation and we can use Proposition \ref{different_characteristics} or \ref{equal_characteristics} to obtain a suitable $F_2\le \GA(2,K)$.
\end{proof}

Finally, we will consider the cases of locally finite fields and one-dimensional spaces.

 \begin{tw}\label{locally_fin}
    Let $K$ be a field and $n\in \mathbb{N}$.
    If $K$ is locally finite or $n=1$, then $K^n$ is not $\GA(n,K)$-paradoxical.
 \end{tw}
    
 \begin{proof}
    Assume that $K$ is a locally finite field.
 It is then easy to show that  $\GA(n,K)$ is a locally finite group (i.e., its every finite subset generates a finite subgroup). It follows from \cite[Theorem 12.4]{W} that any locally finite group is {\em amenable} \cite[Definition 12.1]{W}. By Tarski's theorem \cite[Theorem A.13]{KateJ}, any amenable group admits no paradoxical action.
    Therefore, $K^n$ is not $\GA(n,K)$-paradoxical for any $n\in \mathbb{N}$.

Similarly, if $K$ is any field, the group $\GA(1,K)\cong K^{*}\ltimes K$ is amenable since it is solvable \cite[Theorem 3.2]{KateJ}. Hence, $K$ is not $\GA(1,K)$-paradoxical.
\end{proof}

\section*{Acknowledgements}

The author would like to thank Professor Terence Tao for a helpful explanation, and Professor Wies{\l}aw \'Sliwa for many inspiring suggestions.

\bibliographystyle{siamplain.bst}
\bibliography{references}

% \printbibliography

\end{document}